\documentclass[amssymb,amstex,amsmath,11pt]{article}
\usepackage[margin=1.4in]{geometry}
\usepackage{amssymb,amsmath}              % See geometry.pdf to learn the layout options. There are lots.
\usepackage{amsmath,amssymb,latexsym,amsfonts,url}
\usepackage{graphicx}
\usepackage{amsthm}
\usepackage{amssymb}
\usepackage{color}
\usepackage{epstopdf}
\renewcommand{\epsilon}{\varepsilon}
\theoremstyle{plain}
\newtheorem{thm}{Theorem}[section]
\newtheorem{prop}[thm]{Proposition}
\newtheorem{cor}[thm]{Corollary}
\newtheorem{lem}[thm]{Lemma}
\newtheorem{rem}[thm]{Remark}

\newtheorem*{theorem*}{Theorem}
\newtheorem*{proposition*}{Proposition}
\theoremstyle{definition}
\newtheorem{df}[thm]{Definition}

\theoremstyle{remark}

\newcommand{\Div}{{\rm div}}
\newcommand{\vol}{{\rm Vol}}
\newcommand{\rvol}{{\rm Vol_{\mathbb{R}}}}
\newcommand{\dist}{{\rm dist}}
\newcommand{\area}{{\rm Area}}
\newcommand{\length}{{\rm Length}}

\newcommand{\mvol}{{\rm Vol_M}}

\def\B{{\mathbb B}}

\def\({\left(}
\def\){\right)}

\begin{document}
\begin{title}
{Optimal isoperimetric inequalities for complete proper minimal submanifolds in hyperbolic space}
\end{title}
\begin{author}{Sung-Hong Min \and Keomkyo Seo}\end{author}

\date{\today}

\maketitle

%\vspace{-1cm}
\begin{abstract}
Let $\Sigma$ be a $k$-dimensional complete proper minimal submanifold in the Poincar\'{e} ball model $B^n$ of hyperbolic geometry. If we consider $\Sigma$ as a subset of the unit ball $B^n$ in Euclidean space, we can measure the Euclidean volumes of the given minimal submanifold $\Sigma$ and the ideal boundary $\partial_\infty \Sigma$, say $\rvol(\Sigma)$ and $\rvol(\partial_\infty \Sigma)$, respectively. Using this concept, we prove an optimal linear isoperimetric inequality. We also prove that if $\rvol(\partial_\infty \Sigma) \geq \rvol(\mathbb{S}^{k-1})$, then $\Sigma$ satisfies the classical isoperimetric inequality. By proving the monotonicity theorem for such $\Sigma$, we further obtain a sharp lower bound for the Euclidean volume $\rvol(\Sigma)$, which is an extension of Fraser and Schoen's recent result \cite{FS} to hyperbolic space. Moreover we introduce the M\"{o}bius volume of $\Sigma$ in $B^n$ to prove an isoperimetric inequality via the M\"{o}bius volume for $\Sigma$. \\


\end{abstract}

%%%%%%%%%%%%%%%%%%%%%%%%%%%%%%%%%%%%%%%%%%%%%%%%
\section{Introduction}
%%%%%%%%%%%%%%%%%%%%%%%%%%%%%%%%%%%%%%%%%%%%%%%%

Let $\Sigma \subset \mathbb{R}^k$ be a domain with smooth boundary $\partial \Sigma$. Then the classical isoperimetric inequality says that
\begin{equation} \label{ineq:classical}
k^k \omega_k \vol(\Sigma)^{k-1} \leq \vol(\partial \Sigma)^k ,
\end{equation}
where equality holds if and only if $\Sigma$ is a ball in $\mathbb{R}^k$. Here $\vol(\Sigma)$ and $\vol(\partial \Sigma)$ denote, respectively, the $k$ and $(k-1)$-dimensional Hausdorff measures, and $\omega_k$ is the volume of the $k$-dimensional unit ball $B^k$. As an extension of this classical isoperimetric inequality for domains in Euclidean space, it is conjectured that the inequality (\ref{ineq:classical}) holds for any $k$-dimensional compact minimal submanifold $\Sigma$ of $\mathbb{R}^n$. In 1921, Carleman \cite{Carleman} gave the first partial proof of the conjecture. He proved the isoperimetric inequality for simply connected minimal surfaces in $\mathbb{R}^n$ by using complex function theory. Osserman and Schiffer \cite{OS} proved in 1975 that the isoperimetric inequality holds for doubly connected minimal surfaces in $\mathbb{R}^3$. Two years later Feinberg \cite{Feinberg} generalized to doubly connected minimal surfaces in $\mathbb{R}^n$ for any dimension $n$. In 1984, Li, Schoen, and Yau \cite{LSY} proved that any minimal surface in $\mathbb{R}^3$ with two boundary components (not necessarily doubly connected) satisfies the isoperimetric inequality. Later, Choe \cite{Choe90} extended this inequality to any minimal surface in $\mathbb{R}^n$ with two boundary components. For higher-dimensional minimal submanifolds, the conjecture was proved only for area-minimizing cases, which was due to Almgren \cite{Almgren}.

In hyperbolic space, there is also a sharp isoperimetric inequality for minimal surfaces. In 1992, Choe and Gulliver \cite{CG1} showed that any minimal surface $\Sigma$ with two boundary components in hyperbolic space $\mathbb{H}^n$ satisfies the sharp isoperimetric inequality
\begin{equation} \label{ineq:isop for min surf in H}
4\pi \area (\Sigma) \leq \length (\partial \Sigma)^2 - \area (\Sigma)^2
\end{equation}
with equality if and only if $\Sigma$ is a geodesic ball in a totally geodesic $2$-plane in $\mathbb{H}^n$. However, there are a few results for higher-dimensional minimal submanifolds in hyperbolic space. Yau \cite{Yau}, Choe and Gulliver \cite{CG2} obtained that if $\Sigma$ is a domain in $\mathbb{H}^k$ or a $k$-dimensional minimal submanifold of $\mathbb{H}^n$, then it satisfies the following linear isoperimetric inequality:
\begin{equation*}
(k-1)\vol (\Sigma) \leq \vol (\partial \Sigma).
\end{equation*}
See \cite{Allard, Castillon, HS, MS, Seo, White} for isoperimetric inequalities involving mean curvature in the more general setting of arbitrary submanifolds.

In this paper we obtain optimal isoperimetric inequalities for complete proper minimal submanifolds in hyperbolic space $\mathbb{H}^n$. Recall that a submanifold in $\mathbb{H}^n$ is {\it proper} if the intersection with any compact set in $\mathbb{H}^n$ is always compact. The existence of complete minimal submanifolds in hyperbolic space was established in \cite{Anderson, Anderson 1983, Lin invent, Lin CPAM}. Throughout this paper, $\mathbb{H}^n$ will denote the hyperbolic $n$-space of constant curvature $-1$. Among several models of $\mathbb{H}^n$, we will identify $\mathbb{H}^n$ with the unit ball $B^n$ in $\mathbb{R}^n$ using the Poincar\'{e} ball model. If $ds_{\mathbb{H}}^2$ denotes the hyperbolic metric on $B^n$, then we have the following conformal equivalence of $\mathbb{H}^n$ with $\mathbb{R}^n$:
\begin{equation*}
ds_{\mathbb{H}}^2 = \frac{4}{(1-r^2)^2}ds_{\mathbb{R}}^2,
\end{equation*}
where $ds_{\mathbb{R}}^2$ is the Euclidean metric on $B^n$ and $r$ is the Euclidean distance from the origin. We recall that any $k$-dimensional totally geodesic submanifold in $\mathbb{H}^n$ is a domain on a $k$-dimensional Euclidean sphere meeting $\partial B^n$ orthogonally. The unit sphere $\mathbb{S}^{n-1}= \partial B^n$ is called the {\it sphere at infinity} and denoted by $\partial_\infty \mathbb{H}^n$.

In Section 2, we study the isoperimetric inequality for a $k$-dimensional complete proper minimal submanifold $\Sigma$ in the Poincar\'{e} ball model $B^n$ of hyperbolic space $\mathbb{H}^n$. Obviously, the $k$-dimensional hyperbolic volume of $\Sigma$ is infinite. However, if we consider $\Sigma$ as a subset of the unit ball $B^n$ in Euclidean space, we can measure the $k$-dimensional Euclidean volume of $\Sigma$, which is denoted by $\rvol(\Sigma)$. Similarly, we measure the $(k-1)$-dimensional Euclidean volume of the ideal boundary $\partial_\infty \Sigma :=\overline{\Sigma} \cap \partial_\infty \mathbb{H}^n$, denoted by $\rvol(\partial_\infty \Sigma)$. By using the concept of the Euclidean volume, we obtain an optimal linear isoperimetric inequality for a $k$-dimensional complete proper minimal submanifold in the Poincar\'{e} ball model $B^n$ of $\mathbb{H}^n$.
\begin{theorem*}
Let $\Sigma$ be a $k$-dimensional complete proper minimal submanifold in the Poincar\'{e} ball model $B^n$. Then
\begin{equation*}
\rvol(\Sigma) \leq \frac{1}{k} \rvol(\partial_{\infty} \Sigma),
\end{equation*}
where equality holds if and only if $\Sigma$ is a $k$-dimensional unit ball $B^k$ in $B^n$.
\end{theorem*}
\noindent This theorem can be regarded as a sort of the result obtained in \cite{CG2} and \cite{Yau}. We further obtain an optimal isoperimetric inequality for a $k$-dimensional complete proper minimal submanifold $\Sigma \subset B^n$ with the Euclidean volume $\rvol(\partial_\infty \Sigma)$ bigger than or equal to the Euclidean volume of the $(k-1)$-dimensional unit sphere $\mathbb{S}^{k-1}$
\begin{equation} \label{ineq:new}
k^k \omega_k \rvol(\Sigma)^{k-1} \leq \rvol(\partial_\infty \Sigma)^k ,
\end{equation}
where equality holds if and only if $\Sigma$ is a $k$-dimensional unit ball $B^k$ in $B^n$ (see Corollary \ref{cor:isop ineq}). It would be interesting to find the classes of ambient spaces and submanifolds in them for which the classical isoperimetric inequality (\ref{ineq:classical}) remains valid (see \cite{Osserman} and \cite{White}). In this point of view, the above inequality (\ref{ineq:new}) gives one possible class of submanifolds of $B^n$ satisfying the classical isoperimetric inequality. It should be mentioned that the inequality (\ref{ineq:new}) has no additional volume term which appears in the inequality (\ref{ineq:isop for min surf in H}).

In Section 3, we give a sharp lower bound for the Euclidean volume of a $k$-dimensional complete proper minimal submanifold $\Sigma$ in the Poincar\'{e} ball model $B^n$. In Euclidean space, Fraser and Schoen \cite{FS} recently obtained that if $\Sigma$ is a minimal surface in the unit ball $B^n\subset \mathbb{R}^n$ with (nonempty) boundary $\partial \Sigma \subset \partial B^n$, and meeting $\partial B^n$ orthogonally along $\partial \Sigma$, then
\begin{equation*}
\area (\Sigma) \geq \pi.
\end{equation*}
In the Poincar\'{e} ball model $B^n$, any complete proper minimal submanifold meets the sphere at infinity $\partial_\infty \mathbb{H}^n$ orthogonally because of the maximum principle. Therefore it is natural to ask whether there is also a sharp lower bound for the Euclidean volume of a complete proper minimal submanifold $\Sigma$ in the Poincar\'{e} ball model or not. Although it is not true even for complete totally geodesic submanifolds, we have the following sharp lower bound for the Euclidean volume under the additional hypothesis that $\Sigma$ contains the origin.
\begin{theorem*}
Let $\Sigma$ be a $k$-dimensional complete proper minimal submanifold in the Poincar\'{e} ball model $B^n$. If $\Sigma$ contains the origin in $B^n$, then
\begin{equation*}
\rvol(\Sigma) \geq \omega_k = \rvol (B^k),
\end{equation*}
where equality holds if and only if $\Sigma$ is a $k$-dimensional unit ball $B^k$ in $B^n$.
\end{theorem*}
\noindent In order to prove this theorem, we prove the monotonicity theorem (see Theorem \ref{thm:monotonicity}). In $\mathbb{R}^n$ and $\mathbb{H}^n$, the monotonicity theorem is well-known and has many important applications \cite{Anderson, CM, EWW, Simon}. By measuring the Euclidean volume rather than the hypervolic volume in $B^n$, we have the following theorem.
\begin{theorem*}
Let $\Sigma$ be a $k$-dimensional complete minimal submanifold in $B^n$. If $r$ is the Euclidean distance from the origin, then
the function $\displaystyle{\frac{\rvol(\Sigma\cap B_r)}{r^k}}$ is nondecreasing in $r$ for $0<r<1$, where $B_r$ is the Euclidean ball of radius $r$ centered at the origin.
\end{theorem*}
\noindent Applying this monotonicity theorem, we give an optimal isoperimetric inequality for complete proper minimal submanifolds containing the origin in $B^n$ (see Corollary \ref{cor:isop ineq from area bound}).

Given a $k$-dimensional complete proper submanifold $\Sigma$ in $B^n$, we introduce the M\"{o}bius volumes of $\Sigma$ and $\partial_\infty \Sigma$ in Section 4. We estimate a lower bound for the M\"{o}bius volume of $\partial_\infty \Sigma$ (Theorem \ref{thm:density}). Using this result, we finally obtain an isoperimetric inequality via the M\"{o}bius volume.

%%%%%%%%%%%%%%%%%%%%%%%%%%%%%%%%%%%%%%%%%%%%%%%%%%%%%%%%%%%%%%%%%%%%%%%%%%%%%%%%%%%%
\section{Linear isoperimetric inequality}
%%%%%%%%%%%%%%%%%%%%%%%%%%%%%%%%%%%%%%%%%%%%%%%%%%%%%%%%%%%%%%%%%%%%%%%%%%%%%%%%%%%%
Let $\Sigma$ be a $k$-dimensional minimal submanifold in the $n$-dimensional hyperbolic space $\mathbb{H}^n$. Let $\rho(x)$ be the (hyperbolic) distance from a fixed point $p$ to $x$ in $\mathbb{H}^n$. We recall that the following basic fact about the Laplacian of the distance, which is due to Choe and Gulliver \cite{CG2}.
\begin{lem}[\cite{CG2}] \label{lem:CG}
Let $\Sigma$ be a $k$-dimensional minimal submanifold in $\mathbb{H}^n$. Then the distance $\rho$ satisfies that
\begin{equation*}
\triangle_\Sigma \rho = \coth \rho (k-|\nabla_\Sigma \rho|^2) .
\end{equation*}
\end{lem}

\noindent Let $f$ be a smooth function in $\rho$ on $\Sigma$. Applying the above lemma, we get
\begin{align*}
\triangle_{\Sigma} f &= \Div(\nabla_\Sigma f)  \\
&= f''|\nabla_\Sigma \rho|^2 + f' \triangle_\Sigma \rho  \\
&= f''|\nabla_\Sigma \rho|^2 + f' \coth \rho (k-|\nabla_\Sigma \rho|^2) \\
&= k f' \coth \rho - |\nabla_\Sigma \rho|^2 (f' \coth \rho - f'').
\end{align*}

\noindent For our purpose, we obtain the following useful lemma.

\begin{lem} \label{lem:laplacian}
Let $\Sigma$ be a $k$-dimensional minimal submanifold in $\mathbb{H}^n$. Then
\begin{equation*}
\triangle_{\Sigma} \frac{1}{(1+\cosh \rho)^{k-1}} \leq -\frac{k(k-1)}{(1+\cosh \rho)^k}.
\end{equation*}
Moreover, equality holds at a point $q \in \Sigma$ if and only if $|\nabla_{\Sigma} \rho(q)|=1$.
\end{lem}

\begin{proof}
Define a smooth function $f$ as
\begin{equation*}
f=\frac{1}{(1+\cosh \rho)^{k-1}}.
\end{equation*}
Since
\begin{eqnarray*}
f'&=&-(k-1)\frac{\sinh \rho}{(1+\cosh \rho)^k} \hspace{1cm} \mbox{and} \\
f''&=&-(k-1)\frac{\cosh \rho (1+ \cosh \rho) - k \sinh^2 \rho}{(1+\cosh \rho)^{k+1}},
\end{eqnarray*}
it immediately follows that
\begin{eqnarray*}
f' \coth \rho - f'' &=& -(k-1)\left(\frac{\sinh \rho}{(1+\cosh \rho)^k} \cdot \frac{\cosh \rho}{\sinh \rho} -\frac{\cosh \rho (1+ \cosh \rho) - k \sinh^2 \rho}{(1+\cosh \rho)^{k+1}}\right)\\
&=& -k(k-1)\frac{\sinh^2 \rho}{(1+\cosh \rho)^{k+1}} \leq 0.
\end{eqnarray*}
From Lemma \ref{lem:CG} and the fact that $|\nabla_\Sigma \rho|\leq 1$, it follows that
\begin{align}
\triangle_{\Sigma} \frac{1}{(1+\cosh \rho)^{k-1}} &\leq k f' \coth \rho - (f' \coth \rho - f'') \label{ineq: in lem} \\
&= (k-1) f' \coth \rho + f'' \nonumber \\
&= -\frac{k(k-1)}{(1+\cosh \rho)^k}. \nonumber
\end{align}
It is clear that equality in the above inequality (\ref{ineq: in lem}) holds at a point $q\in \Sigma$ if and only if $|\nabla_\Sigma \rho (q)| = 1$.

\end{proof}

As mentioned in the introduction, Yau \cite{Yau}, Choe and Gulliver \cite{CG2} proved the linear isoperimetric inequality for a domain $\Sigma$ in $\mathbb{H}^k$ or a $k$-dimensional compact minimal submanifold $\Sigma$ in $\mathbb{H}^n$
\begin{equation*}
(k-1)\vol (\Sigma) \leq \vol (\partial \Sigma).
\end{equation*}
Although this linear isoperimetric inequality is not sharp, it may be observed that the inequality is asymptotically sharp. Motivated by this, we consider a $k$-dimensional complete proper minimal submanifold $\Sigma$ in the Poincar\'{e} ball model $B^n$. If we regard $\Sigma$ as a subset of the unit ball $B^n$ in Euclidean space, we can measure the $k$-dimensional Euclidean volume of $\Sigma$ and the $(k-1)$-dimensional Euclidean volume of $\partial_\infty \Sigma$, say $\rvol(\Sigma)$ and $\rvol(\partial_\infty \Sigma)$, respectively. In what follows, we prove an optimal linear isoperimetric inequality in terms of $\rvol(\Sigma)$ and $\rvol(\partial_\infty \Sigma)$.

\begin{thm} \label{thm:linear isop}
Let $\Sigma$ be a $k$-dimensional complete proper minimal submanifold in the Poincar\'{e} ball model $B^n$. Then
\begin{equation}
\rvol(\Sigma) \leq \frac{1}{k} \rvol(\partial_{\infty} \Sigma), \label{ineq:linear isop}
\end{equation}
where equality holds if and only if $\Sigma$ is a $k$-dimensional unit ball $B^k$ in $B^n$.
\end{thm}

\begin{proof}
Let $r(x)$ and $\rho(x)$ be the Euclidean and hyperbolic distance from the origin to $x\in B^n$, respectively. Recall that the distance functions $r$ and $\rho$ satisfy that
\begin{equation}
\rho=\ln \frac{1+r}{1-r} \text{ \ \ and \ \ } r=\tanh \frac{\rho}{2} = \frac{\sinh \rho}{1+ \cosh \rho}  \label{relation: r and rho} .
\end{equation}
Denote by $B_R$ the Euclidean ball of radius $R$ centered at the origin for $0<R<1$. Note that the Euclidean ball $B_R$ can be thought of as the hyperbolic ball $B_{R^*}$ of radius $R^*$ in the Poincar\'{e} ball model $B^n$, where $R^*=\ln \frac{1+R}{1-R}$.

We now consider two kinds of the intersection of $\Sigma$ with the Euclidean ball $B_R$ and the hyperbolic ball $B_{R^*}$ as following. Denote by $\Sigma_R$ the intersection $\Sigma \cap B_R$ (possibly empty) which has the volume form $dV_\mathbb{R}$ induced from the Euclidean metric. Similarly, denote by $\widetilde{\Sigma}_{R^*}$ the intersection $\Sigma \cap B_{R^*}$ which has the volume form $dV_\mathbb{H}$ induced from the hyperbolic metric.
Since
\begin{equation*}
dV_\mathbb{R} = \left( \frac{1-r^2}{2}\right)^k dV_\mathbb{H},
\end{equation*}
the Euclidean volume $\rvol(\Sigma_R)$ can be computed as
\begin{align} \label{eqn:area comparison}
\rvol(\Sigma_{R}) &= \int_{\Sigma_{R}} d V_{\mathbb{R}} \nonumber\\
&= \int_{\Sigma_{R}} \left(\frac{1-r^2}{2} \right)^k d V_{\mathbb{H}} \nonumber\\
&= \int_{\widetilde{\Sigma}_{R^*}} \frac{1}{(1+ \cosh \rho)^k} d V_{\mathbb{H}},
\end{align}
where we used the equation (\ref{relation: r and rho}) in the last equality.

Define a smooth function $f$ on $\Sigma \subset B^n$ by
\begin{equation*}
f=-\frac{1}{k(k-1)} \cdot \frac{1}{(1+\cosh \rho)^{k-1}} .
\end{equation*}
Using the equality (\ref{eqn:area comparison}), Lemma \ref{lem:laplacian}, and the divergence theorem, we obtain
\begin{align}
\rvol(\Sigma_{R}) &= \int_{\widetilde{\Sigma}_{R^*}} \frac{1}{(1+ \cosh \rho)^k} d V_{\mathbb{H}} \nonumber\\
&\leq \int_{\widetilde{\Sigma}_{R^*}} \triangle_\Sigma f dV_{\mathbb{H}} \nonumber \\
&=\int_{\partial \widetilde{\Sigma}_{R^*}} f' \frac{\partial \rho}{\partial \nu} d\sigma_\mathbb{H} \label{ineq:rvol},
\end{align}
where $d\sigma_\mathbb{H}$ denotes the volume form of the boundary $\partial \widetilde{\Sigma}_{R^*}$ induced from the volume form $dV_\mathbb{H}$ of $\widetilde{\Sigma}_{R^*}$ and $\nu$ denotes the outward unit conormal vector. Since
\begin{equation*}
f'=\frac{1}{k}\frac{\sinh \rho}{(1+\cosh \rho)^k} \hspace{1cm} \mbox{and} \hspace{1cm} \frac{\partial \rho}{\partial \nu} \leq 1,
\end{equation*}
the inequality (\ref{ineq:rvol}) becomes
\begin{equation}
\rvol(\Sigma_{R}) \leq \int_{\widetilde{\Sigma}_{R^*}} \frac{1}{k}\frac{\sinh \rho}{(1+\cosh \rho)^k} d\sigma_\mathbb{H}. \label{ineq:rvol without nu}
\end{equation}
Note that $d\sigma_\mathbb{H}$ can be written as
\begin{equation}
d\sigma_\mathbb{H} = \left( \frac{\sinh \rho}{r}\right)^{k-1} d\sigma_\mathbb{R}, \label{eqn:dsigma}
\end{equation}
where $d\sigma_\mathbb{R}$ denotes the volume form of the boundary $\partial \Sigma_R$ in Euclidean space. From the inequality (\ref{ineq:rvol without nu}) and equality (\ref{eqn:dsigma}), it follows that
\begin{align*}
\rvol(\Sigma_{R}) &\leq \int_{\widetilde{\Sigma}_{R^*}} \frac{1}{k}\frac{\sinh \rho}{(1+\cosh \rho)^k} d\sigma_\mathbb{H}\\
&=\int_{\partial \Sigma_R} \frac{1}{k}\left(\frac{\sinh \rho}{1+\cosh \rho}\right)^k \frac{d\sigma_\mathbb{R}}{r^{k-1}}\\
&=\int_{\partial \Sigma_R} \frac{r}{k} d\sigma_\mathbb{R} \\
&=\frac{R}{k}\rvol(\partial \Sigma_R).
\end{align*}
Therefore, letting $R$ tend to $1$, we obtain
\begin{equation} \label{ineq:limit}
\rvol(\Sigma) \leq \frac{1}{k} \rvol(\partial_{\infty} \Sigma).
\end{equation}
Equality holds in (\ref{ineq:limit}) if and only if $|\nabla_\Sigma \rho (q)|=1$ and $\frac{\partial \rho}{\partial \nu}(q') = \langle\nabla_\Sigma \rho (q'), \nu(q')\rangle=1$ for all $0<R<1$, $q\in \widetilde{\Sigma}_{R^*}$, and $q' \in \partial \widetilde{\Sigma}_{R^*}$, which is equivalent to that $\Sigma$ is totally geodesic in $B^n$ and contains the origin. Therefore equality holds if and only if $\Sigma$ is a $k$-dimensional unit ball $B^k$ centered at the origin in $B^n$.

\end{proof}

For a $k$-dimensional complete proper minimal submanifold $\Sigma$ in the Poincar\'{e} ball model $B^n$, if the Euclidean volume $\rvol(\partial_\infty \Sigma)$ of the ideal boundary $\partial_\infty \Sigma$ is greater than or equal to that of $(k-1)$-dimensional unit sphere, then we have the following sharp isoperimetric inequality.
\begin{cor} \label{cor:isop ineq}
Let $\Sigma$ be a $k$-dimensional complete proper minimal submanifold in the Poincar\'{e} ball model $B^n$. If $\rvol(\partial_{\infty} \Sigma) \geq \rvol(\mathbb{S}^{k-1})=k \omega_k$, then
\begin{equation*}
k^k \omega_k \rvol(\Sigma)^{k-1} \leq \rvol(\partial_{\infty} \Sigma)^k,
\end{equation*}
where equality holds if and only if $\Sigma$ is a $k$-dimensional unit ball $B^k$ in $B^n$.
\end{cor}

\begin{proof}
From Theorem \ref{thm:linear isop} and the assumption that $\rvol(\partial_{\infty} \Sigma) \geq \rvol(\mathbb{S}^{k-1})=k \omega_k$, it follows that
\begin{align*}
k^k \omega_k \rvol(\Sigma)^{k-1} &\leq k^k \omega_k \left( \frac{1}{k} \rvol(\partial_\infty \Sigma)\right)^{k-1}\\
&= k \omega_k \rvol(\partial_\infty \Sigma)^{k-1}\\
&\leq \rvol(\partial_\infty \Sigma)^k.
\end{align*}
\end{proof}

\begin{rem}
{\rm The conclusion of Corollary \ref{cor:isop ineq} is sharp in the following sense: Assume that $\Sigma$ is totally geodesic in the Poincar\'{e} ball model $B^n$. We see that $\Sigma$ is the intersection of a sphere $S$ centered $O'$ with the unit ball $B^n$ in Euclidean space. Let $l$ be the straight line joining two centers $O$ and $O'$. Denote by $\theta$ the angle between the line $l$ and the line connecting the origin $O$ and any boundary point of $\Sigma$ (Figure 1). Obviously, the radius of the sphere $S$ is $\tan \theta$. Moreover, the Euclidean volume $\rvol(\Sigma)$ and $\rvol(\partial_\infty \Sigma)$ are given by

\begin{figure}
\begin{center}
\includegraphics[height=7cm, width=7cm]{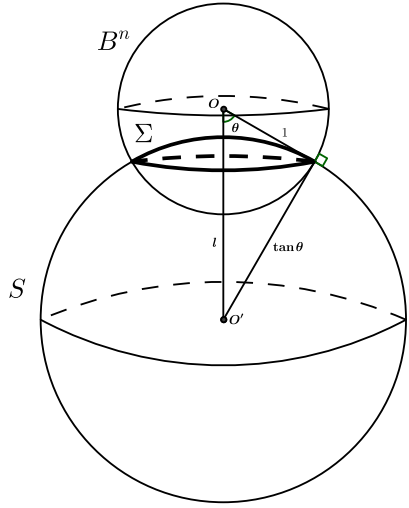}
\end{center}
\caption{a totally geodesic submanifold in $B^n$}
\end{figure}

\begin{align*}
\rvol(\partial_\infty \Sigma) &= \tan^{k-1} \theta  \int_{\mathbb{S}^{k-1}} \sin^{k-1} \left(\frac{\pi}{2}-\theta\right) d\sigma\\
&=k\omega_k \sin^{k-1} \left(\frac{\pi}{2}-\theta\right) \tan^{k-1} \theta \\
&=k\omega_k \sin^{k-1} \theta
\end{align*}
and
\begin{align*}
\rvol(\Sigma) &= \tan^{k} \theta  \int_{\mathbb{S}^{k-1}} \int_0^{\frac{\pi}{2}-\theta} \sin^{k-1} \varphi ~d\varphi d\sigma\\
&=k\omega_k \tan^{k} \theta \int_0^{\frac{\pi}{2}-\theta} \sin^{k-1} \varphi ~ d\varphi  ,
\end{align*}
where $d\sigma$ denotes the volume form of the $(k-1)$-dimensional unit sphere $\mathbb{S}^{k-1}$. Note that the factors $\tan^{k-1} \theta$ and $\tan^{k} \theta$ are the scaling factors of the Euclidean volume of $\partial_\infty \Sigma$ and $\Sigma$, respectively. For such $\Sigma$,
\begin{align*}
k^k\omega_k \rvol(\Sigma)^{k-1} &- \rvol(\partial_\infty \Sigma)^k \\
&= k^k\omega_k \left( k\omega_k \tan^{k} \theta \int_0^{\frac{\pi}{2}-\theta} \sin^{k-1} \varphi ~ d\varphi\right)^{k-1} - (k\omega_k \sin^{k-1} \theta)^k \\
&=k^k{\omega_k}^k \tan^{k(k-1)} \theta ~\left[ \Big(k \int_0^{\frac{\pi}{2}-\theta} \sin^{k-1} \varphi ~ d\varphi \Big)^{k-1} - \cos^{k(k-1)} \theta \right]\\
&=k^{2k-1}{\omega_k}^k \tan^{k(k-1)}\theta ~ \left[ \Big(\int_0^{\frac{\pi}{2}-\theta} \sin^{k-1} \varphi ~ d\varphi \Big)^{k-1} - \left(\frac{\cos^k \theta}{k}\)^{k-1} \right] .
\end{align*}

\noindent Since
\begin{align*}
\int_0^{\frac{\pi}{2}-\theta} \sin^{k-1} \varphi ~ d\varphi &\geq \int_0^{\frac{\pi}{2}-\theta} \sin^{k-1} \varphi \cos \varphi ~ d\varphi = \frac{\cos^k \theta}{k},
\end{align*}
we obtain the {\it reverse} isoperimetric inequality for a totally geodesic submanifold $\Sigma \subset \B^n$ as follows:
\begin{equation*}
k^k \omega_k \rvol(\Sigma)^{k-1} \geq \rvol(\partial_{\infty} \Sigma)^k.
\end{equation*}
Moreover it is not hard to see that equality holds if and only if $\Sigma$ is a $k$-dimensional unit ball $B^k$ in Euclidean space. Hence we conclude that the result of Corollary \ref{cor:isop ineq} is sharp.}
\end{rem}

%%%%%%%%%%%%%%%%%%%%%%%%%%%%%%%%%%%%%%%%%%%%%%%%%%%%%%%%%%%%%%%%%%%%%%%%%%%%%%%%%%%%
\section{Monotonicity}
%%%%%%%%%%%%%%%%%%%%%%%%%%%%%%%%%%%%%%%%%%%%%%%%%%%%%%%%%%%%%%%%%%%%%%%%%%%%%%%%%%%%
One of important properties of a $k$-dimensional minimal submanifold $\Sigma$ in Euclidean or hyperbolic space is the monotonicity: the volume of $\Sigma \cap B_p (r)$ divided by the volume of the $k$-dimensional geodesic ball of radius $r$ is a nondecreasing function of $r$, where $B_p (r)$ is the geodesic ball of radius $r$ with center $p$ in the ambient space for $0<r<\dist (p, \partial \Sigma)$ (see \cite{Anderson, CM, EWW, Simon}).

Let $r(x)$ and $\rho(x)$ be the Euclidean and hyperbolic distance from the origin to $x\in B^n$, respectively. Denote by $B_r$ the Euclidean ball of radius $r$ centered at the origin for $0<r<1$. Recall that the Euclidean ball $B_r$ can be described as the hyperbolic ball $B_\rho$ of radius $\rho$ in the Poincar\'{e} ball model $B^n$, where $\rho=\ln \frac{1+r}{1-r}$. Then the statement of the monotonicity theorem is the following.

\begin{thm}  \label{thm:monotonicity}
Let $\Sigma$ be a $k$-dimensional complete minimal submanifold in $B^n$. Then
the function $\displaystyle{\frac{\rvol(\Sigma\cap B_r)}{r^k}}$ is nondecreasing in $r$ for $0<r<1$. In other words,
\begin{equation*}
\frac{d}{d r} \left( \frac{\rvol(\Sigma \cap B_r)}{r^k} \right) \geq 0,
\end{equation*}
which is equivalent to
\begin{equation*}
\frac{d}{d \rho} \left( \frac{\rvol(\Sigma \cap B_r)}{r^k} \right) \geq 0 .
\end{equation*}
\end{thm}

\begin{proof}
Let $\Sigma_r:=\Sigma \cap B_r$ (possibly empty) be equipped with the Euclidean volume form $dV_\mathbb{R}$. Similarly, let $\widetilde{\Sigma}_{\rho}:= \Sigma\cap B_\rho$ (possibly empty) be equipped with the hyperbolic volume form $dV_\mathbb{H}$. Consider a function $f$ on $\Sigma$ defined as
\begin{equation*}
f=-\frac{1}{k(k-1)} \cdot \frac{1}{(1+\cosh \rho)^{k-1}} .
\end{equation*}
Then Lemma \ref{lem:laplacian} gives us
\begin{align}
\rvol(\Sigma_{r}) &= \int_{\Sigma_{r}} d V_{\mathbb{R}} \nonumber \\
&= \int_{\widetilde{\Sigma}_{\rho}} \frac{1}{(1+ \cosh \rho)^k} d V_{\mathbb{H}} \nonumber \\
&\leq \int_{\widetilde{\Sigma}_{\rho}} \triangle_\Sigma f d V_{\mathbb{H}} \label{ineq:a} .
\end{align}
Applying the divergence theorem, we get
\begin{align*}
\int_{\widetilde{\Sigma}_{\rho}} \triangle_\Sigma f d V_{\mathbb{H}} &= \int_{\partial \widetilde{\Sigma}_{\rho}} f' \frac{\partial \rho}{\partial \nu} d \sigma_{\mathbb{H}}\\
&\leq \int_{\partial \widetilde{\Sigma}_{\rho}} f' ~|\nabla_\Sigma \rho| ~d \sigma_{\mathbb{H}},
\end{align*}
where $d\sigma_\mathbb{H}$ is the hyperbolic volume form of $\partial \widetilde{\Sigma}_{\rho}$ and $\nu$ is the outward unit conormal vector. Since
\begin{align*}
f'=\frac{1}{k}\frac{\sinh \rho}{(1+\cosh \rho)^k},
\end{align*}
we see that
\begin{align}
\int_{\widetilde{\Sigma}_{\rho}} \triangle_\Sigma f d V_{\mathbb{H}} \leq \frac{1}{k} \int_{\partial \widetilde{\Sigma}_{\rho}} \frac{\sinh \rho}{(1+\cosh \rho)^k} ~|\nabla_\Sigma \rho|~ d \sigma_{\mathbb{H}} . \label{ineq:b}
\end{align}
On the other hand, the coarea formula says that
\begin{align}
\frac{d}{d \rho} \int_{\widetilde{\Sigma}_{\rho}} \frac{|\nabla_\Sigma \rho|^2}{(1+\cosh \rho)^k} ~ d V_{\mathbb{H}} =
\int_{\partial \widetilde{\Sigma}_{\rho}} \frac{|\nabla_\Sigma \rho|}{(1+\cosh \rho)^k} ~ d \sigma_{\mathbb{H}} . \label{eqn:coarea}
\end{align}
Therefore combining the inequalities (\ref{ineq:a}), (\ref{ineq:b}) and the equality (\ref{eqn:coarea}), we obtain
\begin{align*}
\rvol(\Sigma_r) %&\leq \int_{\widetilde{\Sigma}_{\rho}} \triangle_\Sigma f d V_{\mathbb{H}}\\
%&\leq \int_{\partial \widetilde{\Sigma}_{\rho}} f'  ~|\nabla_\Sigma \rho| ~d \sigma_{\mathbb{H}}\\
&\leq \frac{1}{k} \int_{\partial \widetilde{\Sigma}_{\rho}}\frac{\sinh \rho}{(1+\cosh \rho)^k} ~|\nabla_\Sigma \rho|~ d \sigma_{\mathbb{H}} \\
&=\frac{\sinh \rho}{k} \int_{\partial \widetilde{\Sigma}_{\rho}}\frac{|\nabla_\Sigma \rho|}{(1+\cosh \rho)^k} ~ d \sigma_{\mathbb{H}}\\
&=\frac{\sinh \rho}{k} \frac{d}{d \rho} \int_{\widetilde{\Sigma}_{\rho}} \frac{|\nabla_\Sigma \rho|^2}{(1+\cosh \rho)^k} ~ d V_{\mathbb{H}}  \\
&=\frac{\sinh \rho}{k}  \left[\frac{d}{d \rho} \int_{\widetilde{\Sigma}_{\rho}} \frac{1}{(1+\cosh \rho)^k} ~ d V_{\mathbb{H}} -
\frac{d}{d \rho} \int_{\widetilde{\Sigma}_{\rho}} \frac{1-|\nabla_\Sigma \rho|^2}{(1+\cosh \rho)^k} ~ d V_{\mathbb{H}} \right] \\
&\leq \frac{\sinh \rho}{k} \frac{d}{d \rho} \int_{\widetilde{\Sigma}_{\rho}} \frac{1}{(1+\cosh \rho)^k} ~ d V_{\mathbb{H}} .
\end{align*}
Since
\begin{align*}
\int_{\widetilde{\Sigma}_{\rho}} \frac{1}{(1+\cosh \rho)^k} ~ d V_{\mathbb{H}} = \int_{\Sigma_r} dV_\mathbb{R},
\end{align*}
we finally get
\begin{align*}
\rvol(\Sigma_r) &\leq \frac{\sinh \rho}{k} \frac{d}{d \rho} \int_{\widetilde{\Sigma}_{\rho}} \frac{1}{(1+\cosh \rho)^k} ~ d V_{\mathbb{H}} \\
&=\frac{\sinh \rho}{k} \frac{d}{d \rho} \int_{\Sigma_r} dV_\mathbb{R}\\
&=\frac{\sinh \rho}{k} \frac{d}{d \rho} \rvol (\Sigma_r)
\end{align*}
for any $0<r<1$. This implies that
\begin{align*}
\frac{d}{d \rho} \left( \frac{\rvol(\Sigma_r)}{r^k} \right) &= \frac{d}{d \rho} \left( \frac{\rvol(\Sigma_r)}{\tanh^k \frac{\rho}{2}} \right)\\
&=\frac{1}{\tanh^{k+1} \frac{\rho}{2}} \left[ \tanh \frac{\rho}{2}\cdot \frac{d}{d \rho}\rvol(\Sigma_r) - \frac{k\rvol(\Sigma_r)}{2\cosh^2 \frac{\rho}{2}} \right]\\
&=\frac{1}{\tanh^{k+1} \frac{\rho}{2}} \left[ \frac{\sinh \rho}{1+\cosh \rho}\cdot \frac{d}{d \rho}\rvol(\Sigma_r) - \frac{k\rvol(\Sigma_r)}{1+\cosh \rho} \right]\\
&= \frac{1}{\tanh^{k+1} \frac{\rho}{2}} \cdot \frac{1}{1+\cosh \rho} \left[ \sinh \rho \cdot \frac{d}{d \rho}\rvol(\Sigma_r) - k\rvol(\Sigma_r) \right] \\
&\geq 0,
\end{align*}
which completes the proof.

\end{proof}
Recall that the {\it density} $\Theta (\Sigma, p)$ of a $k$-dimensional submanifold $\Sigma$ in a Riemannian manifold $M$ at a point $p\in M$ is defined to be
\begin{eqnarray*}
\Theta (\Sigma, p) = \lim_{\varepsilon \rightarrow 0}\frac{\vol
(\Sigma \cap B_\varepsilon (p))}{\omega_k \varepsilon^k},
\end{eqnarray*}
where $B_\varepsilon (p)$ is the geodesic ball of $M$ with radius $\varepsilon$ and center $p$. As a consequence of Theorem \ref{thm:monotonicity}, we have
\begin{cor} \label{cor:isop ineq from area bound}
Let $\Sigma$ be a $k$-dimensional complete proper minimal submanifold containing the origin in the Poincar\'{e} ball model $B^n$. Then
\begin{equation*}
\rvol(\Sigma) \geq \omega_k = \rvol(B^k).
\end{equation*}
\end{cor}

\begin{proof}
Since the function $\frac{\rvol(\Sigma_r)}{r^k}$ is nondecreasing in $r$ by Theorem \ref{thm:monotonicity},
\begin{align*}
\rvol(\Sigma) = \lim_{r\rightarrow 1-} \frac{\rvol(\Sigma_r)}{r^k} \geq \lim_{r\rightarrow 0+} \frac{\rvol(\Sigma_r)}{r^k} = \omega_k \Theta(\Sigma, O) \geq \omega_k.
\end{align*}
%where $\Theta(\Sigma, O)$ denotes the density of $\Sigma$ at the origin $O$.
\end{proof}

Suppose that a $k$-dimensional complete proper minimal submanifold $\Sigma$ contains the origin in the Poincar\'{e} ball model. Combining Theorem \ref{thm:linear isop} and Corollary \ref{cor:isop ineq from area bound}, we further obtain
\begin{align*}
\rvol(\partial_\infty \Sigma) &\geq k \rvol(\Sigma)\\
&\geq k \rvol (B^k) \\
&= \rvol (\mathbb{S}^{k-1}).
\end{align*}
Thus we see that $\rvol(\partial_{\infty} \Sigma) \geq \rvol(\mathbb{S}^{k-1})=k \omega_k$. Hence the following can be derived from Corollary \ref{cor:isop ineq}.
\begin{cor} \label{cor:isop origin}
Let $\Sigma$ be a $k$-dimensional complete proper minimal submanifold containing the origin in the Poincar\'{e} ball model $B^n$. Then
\begin{equation*}
k^k \omega_k \rvol(\Sigma)^{k-1} \leq \rvol(\partial_{\infty} \Sigma)^k,
\end{equation*}
where equality holds if and only if $\Sigma$ is a $k$-dimensional unit ball $B^k$ in $B^n$.
\end{cor}

%%%%%%%%%%%%%%%%%%%%%%%%%%%%%%%%%%%%%%%%%%%%%%%%%%%%%%%%%%%%%%%%%%%%%%%%%%%%%%%%%%%%
\section{M\"{o}bius volume}
%%%%%%%%%%%%%%%%%%%%%%%%%%%%%%%%%%%%%%%%%%%%%%%%%%%%%%%%%%%%%%%%%%%%%%%%%%%%%%%%%%%%

In \cite{Min}, the first author introduced the M\"{o}bius volume of a compact submanifold of $\mathbb{S}^{n-1} = \partial_\infty \mathbb{H}^n$ as follows.
\begin{df} \label{df:Min}
Let $\Gamma$ be a compact submanifold of
$\mathbb{S}^{n-1}$. Let {M{\"{o}}b}($\mathbb{S}^{n-1}$) be the group of all {M{\"{o}}bius}
transformations of $\mathbb{S}^{n-1}$. The {\it M{\"{o}}bius volume} $\mvol(\Gamma)$ of $\Gamma$
is defined to be
\begin{equation*}
\mvol(\Gamma)=\sup\{{\vol}_\mathbb{R} (g(\Gamma))
\,|\, {g \in {\text{\rm M{\"{o}}b}}(\mathbb{S}^{n-1})}\}.
\end{equation*}
\end{df}
\noindent According to the definition of Li and Yau \cite{LY}, the M{\"{o}}bius volume of $\Gamma$ is the same as the $(n-1)$-conformal volume of the inclusion of $\Gamma$ into $\mathbb{S}^{n-1}$. Using this notation, the first author was able to prove the embeddedness of a complete proper minimal submanifold in hyperbolic space. In particular, he gave a lower bound for the M\"{o}bius volume of a compact submanifold $\Gamma \in \mathbb{S}^{n-1}$.
\begin{prop}[\cite{Min}] \label{thm:min}
Let $\Gamma$ be a $k$-dimensional compact submanifold of $\mathbb{S}^{n-1}$. Then
\begin{equation*}
\mvol(\Gamma)\geq \rvol(\mathbb{S}^k),
\end{equation*}
where equality holds if $\Gamma$ is a $k$-dimensional unit sphere $\mathbb{S}^k$.
\end{prop}

Given a $k$-dimensional complete proper minimal submanifold $\Sigma \subset B^n$, we estimate a lower bound for the M\"{o}bius volume of the ideal boundary $\partial_\infty \Sigma$.
\begin{thm} \label{thm:density}
Let $\Sigma$ be a $k$-dimensional complete proper minimal submanifold in $B^n$. Then
\begin{equation*}
\mvol(\partial_\infty \Sigma)\geq \rvol(\mathbb{S}^{k-1}) \cdot \max_{p\in {\Sigma}} \Theta (\Sigma, p) .
\end{equation*}
\end{thm}
\begin{proof}
Since $\Sigma$ is proper in $B^n$, $\max\limits_{p\in {\Sigma}} \Theta (\Sigma, p) $ is finite. Moreover the maximum is attained in $\Sigma$ because the density $\Theta (\Sigma, p)$ is integer-valued there. Now we may assume that $\max\limits_{p\in {\Sigma}} \Theta (\Sigma, p)$ is attained at $q\in \Sigma$. Take an isometry $\varphi$ of hyperbolic space $\mathbb{H}^n$ such that $\varphi(q)=O$. Since the group of all isometries of $\mathbb{H}^n$ is isomorphic to {M{\"{o}}b}($\mathbb{S}^{n-1}$), we may consider $\varphi$ as an element of {M{\"{o}}b}($\mathbb{S}^{n-1}$) (see \cite{Ratcliffe}). Then by applying Theorem \ref{thm:linear isop} and Theorem \ref{thm:monotonicity}, we see that
\begin{align*}
\mvol(\partial_\infty \Sigma) &\geq \rvol(\partial_\infty \varphi(\Sigma))\\
&\geq k \rvol (\varphi(\Sigma)) \\
&\geq k \omega_k \Theta(\varphi(\Sigma), O) \\
&= \rvol(\mathbb{S}^{k-1}) \Theta(\Sigma, q),
\end{align*}
where we used the invariance of the density under an isometry of hyperbolic space in the last equality. This completes the proof.

\end{proof}

Since $\max\limits_{ p\in \Sigma}\Theta (\Sigma, p) \geq 1$, Theorem \ref{thm:density} gives another proof of the above Proposition \ref{thm:min} for $\Gamma = \partial_\infty \Sigma$.
\begin{cor} \label{cor:mvol}
Let $\Sigma$ be a $k$-dimensional complete proper minimal submanifold in $B^n$. Then
\begin{equation*}
\mvol(\partial_\infty \Sigma)\geq \rvol(\mathbb{S}^{k-1}).
\end{equation*}
\end{cor}
Furthermore if $\Sigma$ has a self-intersection point $p$ in $\Sigma$, then the density $\Theta (\Sigma, p) \geq 2$. Therefore
\begin{cor}
Let $\Sigma$ be a $k$-dimensional complete proper minimal submanifold in $B^n$. If $\Sigma$ has a self-intersection point in $\Sigma\setminus \partial_\infty \Sigma$, then
\begin{equation*}
\mvol(\partial_\infty \Sigma)\geq 2\rvol(\mathbb{S}^{k-1}).
\end{equation*}
\end{cor}

In this section, we introduce the M\"{o}bius volume of a submanifold in $B^n$ as in the Definition \ref{df:Min}.
\begin{df}
Let $\Sigma$ be a submanifold in the Poincar\'{e} ball model $B^n$. Let {M{\"{o}}b}($B^n$) be the group of all {M{\"{o}}bius}
transformations of $B^n$. The {\it M{\"{o}}bius volume} $\mvol(\Sigma)$ of $\Sigma$
is defined to be
\begin{equation*}
\mvol(\Sigma)=\sup\{{\vol}_\mathbb{R} (g(\Sigma))
\,|\, {g \in {\text{\rm M{\"{o}}b}}(B^n)}\}.
\end{equation*}
\end{df}
Note that since {M{\"{o}}b}($\mathbb{S}^{n-1}$) is isomorphic to {M{\"{o}}b}($B^n$), any element of {M{\"{o}}b}($\mathbb{S}^{n-1}$) can be considered as an element of {M{\"{o}}b}($B^n$). Using this concept, we obtain an isoperimetric inequality for any complete proper minimal submanifold in hyperbolic space with no assumption on $\Sigma$ unlike Corollary \ref{cor:isop ineq} and Corollary \ref{cor:isop origin}.
\begin{thm}
Let $\Sigma$ be a $k$-dimensional complete proper minimal submanifold in $B^n$. Then
\begin{equation}
k^k \omega_k \mvol(\Sigma)^{k-1} \leq \mvol(\partial_{\infty} \Sigma)^k. \label{ineq:weak mvol}
\end{equation}
\end{thm}
\begin{proof}
From the linear isoperimetric inequality (\ref{ineq:linear isop}), it follows that for any isometry $\varphi$ of $\mathbb{H}^n$,
\begin{equation*}
\rvol(\varphi(\Sigma)) \leq \frac{1}{k}\rvol(\partial_{\infty} \varphi(\Sigma)).
\end{equation*}
Therefore by the definition of the M{\"{o}}bius volume
\begin{equation}
\mvol(\Sigma) \leq \frac{1}{k}\mvol(\partial_{\infty} \Sigma). \label{ineq:mobius volume}
\end{equation}
Corollary \ref{cor:mvol} and the inequality (\ref{ineq:mobius volume}) yield that
\begin{align*}
k^k \omega_k \mvol(\Sigma)^{k-1} &\leq k^k \omega_k \left( \frac{1}{k} \mvol(\partial_\infty \Sigma)\right)^{k-1}\\
&= k \omega_k \mvol(\partial_\infty \Sigma)^{k-1}\\
&= \rvol(\mathbb{S}^{k-1}) \mvol(\partial_\infty \Sigma)^{k-1}\\
&\leq \mvol(\partial_\infty \Sigma)^k,
\end{align*}
which completes the proof.

\end{proof}

\begin{rem}
{\rm We see that if $\Sigma$ is a $k$-dimensional complete totally geodesic submanifold in $B^n$, then equality holds in (\ref{ineq:weak mvol}). However we do not know whether the converse is true.
}
\end{rem}

\vskip 0.3cm
\noindent
{\bf Acknowledgment: }The authors would like to thank Professor Jaigyoung Choe for his valuable comments on this work.
%%%%%%%%%%%%%%%%%%%%
%%%%%%%%%%%%%%%%%%%%
%%%%%%%%%%%%%%%%%%%%
%\addcontentsline{toc}{section}{3 \hspace{0.09cm} References}

\vskip 1cm
\noindent Sung-Hong Min\\
Korea Institute for Advanced Study, 207-43 Cheongnyangni 2-Dong, Dongdaemun-Gu, Seoul 130-722, Korea\\
{\tt e-mail:shmin@kias.re.kr}\\

\bigskip
\noindent Keomkyo Seo\\
Department of Mathematics, Sookmyung Women's University, Hyochangwongil 52, Yongsan-ku, Seoul, 140-742, Korea\\
{\tt E-mail:kseo@sookmyung.ac.kr}\\
URL: http://sookmyung.ac.kr/$\thicksim$kseo
\end{document}